\documentclass{article}
\usepackage{enumitem}

\usepackage{lineno}

\usepackage{amsfonts,amsmath,amsthm,amscd,amssymb,latexsym,amsbsy,pb-diagram,cite,caption,url}

\usepackage{tikz}
\usetikzlibrary{positioning, calc, trees, decorations.markings, patterns, arrows, arrows.meta}
\usetikzlibrary{trees}

\usepackage{diagmac2}
\usepackage{float}
\usepackage{circledtext}

\newtheorem{theorem}{Theorem}[section]
\newtheorem{lemma}[theorem]{Lemma}
\newtheorem{corollary}[theorem]{Corollary}
\newtheorem{proposition}[theorem]{Proposition}

\newtheorem{conjecture}[theorem]{Conjecture}

\theoremstyle{definition}
\newtheorem{definition}[theorem]{Definition}

\theoremstyle{remark}
\newtheorem{example}[theorem]{Example}
\newtheorem{remark}[theorem]{Remark}

\newcommand{\keywords}{\textbf{Key words. }\medskip}
\newcommand{\subjclass}{\textbf{MSC 2020. }\medskip}
\renewcommand{\abstract}{\textbf{Abstract. }\medskip}

\numberwithin{equation}{section}

\begin{document}

\author{Oleksiy Dovgoshey, Omer Cantor, Olga Rovenska}

\title{Compact ultrametric spaces generated by labeled star graphs}

\maketitle

\begin{abstract}
 Let ${\bf US}$ be the class of all ultrametric spaces generated by labeled star graphs. We prove that compact ${\bf US}$-spaces are the completions of totally bounded ultrametric spaces generated by decreasingly labeled rays.  
We characterize the ultrametric spaces which are weakly similar to finite ${\bf US}$-spaces and describe these spaces by certain four-point conditions.  

\end{abstract}



\section{Introduction}

\hspace{4 mm} As is known, any finite ultrametric space can be described up to isometry by a Gurnvich–Vyalyi representing tree \cite{Gurvich-Vyalyi2012}. 
This representation and its geometric interpretation \cite{Petrov-Dovgoshey2014} allow us to find solutions of different extremal problems related to such spaces \cite{Dovgoshey-Petrov2020, Dovgoshey-Petrov-Teichert2015, Dovgoshey-Petrov-Teichert2017}. 
Recently, analogues of the Gurnvich–Vyalyi  representation  and its geometric interpretation  were also found for totally bounded ultrametric spaces in \cite{Dov}.

The above mentioned representing  trees form a special class of  trees endowed with vertex labelings. Some important properties of infinite trees endowed with positive edge labelings have been described in \cite{BD2006CPC, BDS2005JGT, BS2010CPC, DIESTEL2006846, DIESTEL20111423, Deistel2017, DK2004EJC, DP2017JCT, DS2011AM, DS2011TIA, DS2012DM, DMV2013AC, DP2013SM}.

The ultrametric spaces generated by arbitrary nonnegative vertex labelings on both finite and infinite trees   were first considered in \cite{Dov2020TaAoG} and studied in \cite{Dovgoshey-Küçükaslan2022,Dovgoshey-Kostikov-2023}. The simplest types of infinite trees are rays and star graphs. The totally bounded ultrametric spaces generated by labeled  almost rays have  been characterized in \cite{Dovgoshey-Vito}. 
Furthermore, paper \cite{Dov-Rov} contains a purely metric characterization of ultrametric spaces generated by labeled star graphs.

The main goal of this paper is to give a metric description of compact ultrametric spaces generated by labeled star graphs.

The paper is organized as follows. In Section~2 we collect together some definitions and facts related to ultrametric spaces and trees.  

Section~3 contains the formulations of two problems that initially motivate our study.  

The finite ultrametric spaces generated by labeled stars are considered in Section~4.  
 Theorem~\ref{eeerlm} shows that Conjecture 4.1 of \cite{Dov-Rov} is true.
A semimetric modification of Theorem~\ref{eeerlm} is given in Theorem~\ref{333sss}.  

Theorem~\ref{qyuiz}, the first result of Section~5, describes up to isomorphism all labeled star graphs that generate compact ultrametric space.  
Compact ultrametric spaces generated by labeled star graphs are described up to isometry in Theorem~\ref{siti}.  
In Theorem~\ref{wop} we prove that Conjecture~4.2 of paper \cite{Dov-Rov} is true.  

Our last result, Theorem~\ref{sofas}, shows that the completions of totally bounded ultrametric spaces generated by decreasingly labeled rays coincide with compact ultrametric spaces generated by labeled star graphs.

The final Section~6 contains two new conjectures.

\section{Basic definitions and preliminary results}

Let us start from the fundamental concept of semimetric space introduced by M.~Fréchet in \cite{Fréchet}.

\begin{definition}\label{wdvthk}
Let $\mathbb{R}^+ = [0, \infty)$. 
A {\it semimetric} on a non-empty set $X$ is a symmetric function
$
d \colon X \times X \to \mathbb{R}^+
$
such that $d(x, y) = 0$ if and only if $x = y$.  A semimetric $d$ is called an \textit{ultrametric} if the {\it strong triangle inequality}
\begin{equation*}
    d(x, y) \leq \max\{d(x, z),\,d(z,y)\}
\end{equation*}
      holds for all $x, y, z \in X$.
\end{definition}

The object of our study is a certain family of compact, and in particular finite, ultrametric spaces.
A standard definition of compactness is usually formulated as: A subset  $ S$ of an ultrametric space   is compact if each open cover of $ S $  has a finite subcover.

There exists a simple interdependence between the compactness of a set and the so-called convergent sequences. Let us remember that a sequence $(x_n)_{n \in \mathbb{N}}$ of points in an ultrametric space $(X, d)$ is said to converge to a point $a \in X$ if
\begin{equation}
\label{896}
    \lim_{n \to \infty} d(x_n, a) = 0.
\end{equation}

When \eqref{896} holds, we write  
\begin{equation*}
\lim_{n \to \infty} x_n = a.
\end{equation*}

Recall that a point $a \in X$ is a \textit{limit point} (or equivalently, an \textit{accumulation point}) of a set $A \subseteq X$ if there is a sequence $(a_n)_{n \in \mathbb{N}}$ of different points of $A$ such that  
\begin{equation*}
\lim_{n \to \infty} a_n = a.
\end{equation*}

\begin{proposition}\label{graa}
A subset $A$ of an ultrametric space is compact if and only if every infinite sequence of points of $A$ contains a subsequence which converges to a point of $A$.
\end{proposition}

For the proof of Proposition~\ref{graa}  see, for example,  Theorem 12.1.3 in \cite{Sear}.

Let $(X,d)$ be a semimetric space. An \textit{open ball} with a \textit{radius} $r > 0$ and a \textit{center} $c \in X$ is the set  
\begin{equation*}
B_r(c) = \{ x \in X : d(c,x) < r \}.
\end{equation*}

A subset $O$ of an ultrametric space $(X, d)$ is called {\it open} if for every point $p \in O$ there is an open ball $B$ such that  
\begin{equation*}  
p \in B \subseteq O.  
\end{equation*}

\begin{proposition}\label{5667dvf}
Let $(X,d)$ be a compact ultrametric space.  
Then, for every open subset $O$ of $X$, the set $X \setminus O$ is also compact.  
\end{proposition}
The last proposition  follows directly from Theorem 12.2.3 of \cite{Sear}.

\begin{definition} A subset $A$ of an ultrametric space $(X,d)$ is called \textit{totally bounded} if for every $r > 0$ there is a finite set $\{ x_1, \dots, x_n \}\subseteq X$ such that  
\begin{equation*}
A \subseteq \bigcup_{i=1}^{n} B_r(x_i).
\end{equation*}
\end{definition}

Recall also that a sequence $(x_n)_{n \in \mathbb{N}}$ of points in an ultrametric space $(X,d)$ is a {\it Cauchy sequence} iff
\begin{equation*}
    \lim\limits_{\substack{n \to \infty \\ m \to \infty}} d(x_n, x_m) = 0.
\end{equation*}

\begin{proposition}\label{erra}
    A subset $A$ of an ultrametric space $(X, d)$ is totally bounded if and only if every infinite sequence of points in $A$ contains a Cauchy subsequence.
\end{proposition}
\noindent See, for example, Theorem 7.8.2 in \cite{Sear}.

Let $(X,d)$ and $(Y,\rho)$ be semimetric spaces. Recall that a bijective mapping $\Phi: X \to Y$ is an \textit{isometry} if the equality  
\begin{equation*}  
d(x,y) = \rho(\Phi(x), \Phi(y))  
\end{equation*}  
holds for all $x,y \in X$.

Let $(X, d)$ be an ultrametric space and let $A \subseteq X$. A subset $S$ of $A$ is said to be {\it dense} in $A$ if for every $a \in A$ there is a sequence $(s_n)_{n \in \mathbb{N}}$ of points of $S$ such that  
\begin{equation*}
a = \lim\limits_{n \to \infty} s_n.
\end{equation*}

    An ultrametric space $(X,d)$ is \textit{complete} if every Cauchy sequence of points of $X$ converges to a point of $X$.

\begin{definition}\label{ssmmcc}
    Let $(X,d)$ be an ultrametric space.  
An ultrametric space $(Y,\rho)$ is called a \textit{completion} of $(X,d)$ if $(Y,\rho)$ is complete and $(X,d)$ is isometric to dense subspace of $(Y,\rho)$.
\end{definition}


\begin{proposition}\label{qru}
    The completion $(Y, \rho)$ of an ultrametric space $(X,d)$ is compact if and only if $(X,d)$ is totally bounded.
\end{proposition}

For the proof, see, for example, Corollary 4.3.30 in \cite{Engelking}.

Let $(X, d)$ be a semimetric space. Below we denote by $D(X)$ the set of all distances between points of $(X, d)$,
\begin{equation}
\label{ew}
    D(X) := \{d(x, y) : x, y \in X\}.
\end{equation}

The next definition gives us a generalization of the concept of isometry.

\begin{definition}
\label{scguite}
Let $(X, d)$ and $(Y, \rho)$ be semimetric spaces.
 A bijective mapping $\Phi\colon X \to Y$ is a {\it weak similarity} if there is a strictly increasing bijection $f\colon D(Y) \to D(X)$ such that the equality
\begin{equation*}
d(x, y) = f \left( \rho \left( \Phi(x), \Phi(y) \right) \right)
\end{equation*}
holds for all $x, y \in X$.
\end{definition}

\begin{remark}
The notion of weak similarity was introduced in \cite{Dovgoshey-Petrov}.
\end{remark}

We will say that semimetric spaces $(X,d)$ and $(Y,\rho)$ are weakly similar if there is weak similarity  $
X \to Y  $.

The next proposition follows directly from Definition~\ref{wdvthk} and Definition~\ref{scguite}.

\begin{proposition}\label{ooyggr}
Let $(X,d)$ be an ultrametric space and $(Y,\rho)$ be a semimetric space. If $(X,d)$ and $(Y,\rho)$
 are weakly similar, then $(Y,\rho)$ is also an ultrametric space.
\end{proposition}

Let us now recall  some concepts from graph theory.

A \textit{simple graph} is a pair $(V, E)$ consisting of a non-empty set $V$ and a set $E$ whose elements are unordered pairs $\{u, v\}$ of different points $u, v \in V$. For a graph $G = (V, E)$, the sets $V = V(G)$ and $E = E(G)$ are called the \textit{set of vertices} and the \textit{set of edges}, respectively. A graph $G$ is \textit{finite} if $V(G)$ is a finite set. A graph $H$ is, by definition, a \textit{subgraph} of a graph $G$ if the inclusions $V(H) \subseteq V(G)$ and $E(H) \subseteq E(G)$ hold. In this case, we simply write $H \subseteq G$.

In what follows, we will use the standard definitions of paths and cycles, see, for example, \cite[Section 1.3]{Diestel2017}. A graph $G$ is {\it connected} if for every two distinct $u, v \in V(G)$ there is a path $P$ joining $u$ and $v$ in $G,$
\begin{equation*}
    u,\, v\in V(P) \,\, \text{and}\,\, P\subseteq G.
\end{equation*}

A connected graph without cycles is called a {\it tree}.

\begin{definition}
\label{wer}
A tree $T$ is a {\it star graph} if there is a vertex $c \in V(T)$, the {\it center} of $T$, such that $c$ and $v$ are adjacent for every $v \in V(T) \setminus \{c\}$ but for all $u, w \in T$ we have $\{u, w\} \notin E(T)$ whenever $u\neq c$ and $w\neq c.$
\end{definition}

An infinite graph $R$ of the form  
\begin{align}
\label{ktas}
V(R) &= \{x_1, x_2, \ldots, x_n, x_{n+1}, \ldots\}, \\  
\label{ktas1}
E(R) &= \{\{x_1, x_2\}, \ldots, \{x_n, x_{n+1}\}, \dots\},
\end{align}
where all $x_n$ are assumed to be distinct, is called a \textit{ray}~\cite{Diestel2017}.  
If \eqref{ktas} and \eqref{ktas1} hold, we will write $R=(x_1,x_2,\ldots, x_n,\ldots).$
It is clear that every ray is a tree.

In what follows, by \textit{labeled} tree $T(l)$ we will mean a tree $T$ equipped with a labeling  $l\colon V(T) \to \mathbb{R}^+.$

Let $T(l)$ be a labeled tree. As in~\cite{Dov2020TaAoG} we define a mapping $d_l \colon V(T) \times
V(T) \to \mathbb{R}^{+}$ by
\begin{equation}  \label{e1.1}
d_l(u, v) =
\begin{cases}
0, & \text{if } u = v, \\
\max\limits_{w \in V(P)} l(w), & \text{otherwise},%
\end{cases}%
\end{equation}
where $P$ is the path joining $u$ and $v$ in $T$.

The following result is a direct corollary of Proposition\,3.2 from \cite{Dov2020TaAoG}.

\begin{theorem}\label{t1.4}
Let $T = T(l)$ be a labeled tree. Then the
function $d_l$ is an ultrametric on $V(T)$ if and only if the inequality
\begin{equation*}
\max\{l(u), l(v)\} > 0
\end{equation*}
holds for every edge $\{u, v\}$ of $T$.
\end{theorem}


In what follows, we also will say that an ultrametric space $(X, d)$ is \textit{generated by labeled tree} $T(l)$ if the equalities
$
X = V(T) $ and $d = d_l$
hold, where $d_l$ is defined as in \eqref{e1.1}.

\begin{definition}
    \label{thukm}
An ultrametric space $(X, d)$ belongs to the  class ${\bf US}$ if $(X, d)$ is generated by some labeled star graph.
\end{definition}

The following theorem was proved in \cite{Dov-Rov}.

\begin{theorem}
    \label{xz}
Let $(X, d)$ be an ultrametric space. Then the following statements are equivalent:
\begin{enumerate}[label=\normalfont(\roman*), left=0pt] 
\item $(X, d) \in {\bf US}.$

\item There is $x_0 \in X$ such that 
\begin{equation}
\label{zg}
    d(x_0, x) \leq d(y,x)
\end{equation}
whenever $x,y\in X$ and
\begin{equation}
\label{cd}
    x_0 \neq x \neq y.
\end{equation}
\end{enumerate}

\end{theorem}

The next proposition can be considered as a partial refinement of Theorem~\ref{xz}.

\begin{proposition}\label{fraas}
Let $(X,d)$ be an ultrametric space and let $x_0 \in X$. If, for $x,y \in X$, inequality \eqref{zg} holds whenever we have \eqref{cd}, then $(X,d)$ is generated by labeled star graph $S(l)$ with center $x_0$ and labeling $l: X \to \mathbb{R}^+$ defined as  
\begin{equation*}
    l(x) = d(x,x_0)
\end{equation*}
for each $x \in X$.  
\end{proposition} 

The proof of this proposition coincides with the second part of the proof of Theorem 2.1 in \cite{Dov-Rov}.

\section{Motivating problems}

\hspace{4 mm} The four-point ultrametric spaces $(X_4, d_4)$ and $(Y_4, \rho_4)$  
depicted in Figure \ref{cis} below do not belong to ${\bf US}$.

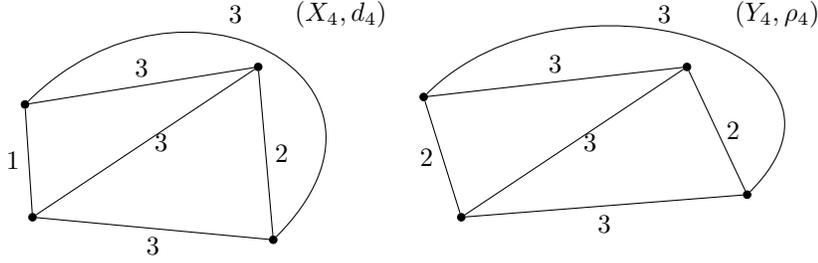
\begin{figure}[H]
\centering
\begin{tikzpicture}[remember picture]

  \node[draw=none] (label1) at (2.7,2.7) {$3$};
  \node[draw=none] (label1) at (4.1,2.7) {$(X_4,d_4)$};
    
    \node[draw, circle, fill=black, inner sep=1pt] (A) at (0,0) {};
    \node[draw, circle, fill=black, inner sep=1pt] (B) at (3.2,-0.3) {};
    \node[draw, circle, fill=black, inner sep=1pt] (C) at (-0.1,1.5) {};
    \node[draw, circle, fill=black, inner sep=1pt] (D) at (3,2) {};

    \draw (A) -- (B) node[midway, below] {$3$};
    \draw (A) -- (C) node[midway, left] {$1$};
    \draw (A) -- (D) node[midway, right] {$3$};
    \draw (C) -- (D) node[midway, above] {$3$};
    \draw (B) -- (D) node[midway, right] {$2$};
    \draw (B) to[out=45,in=45,looseness=2] (C) node[midway,right] {};

  \begin{scope}[xshift=5.7cm]
    \node[draw=none] (label2) at (2.7,2.7) {$3$};
    \node[draw=none] (label2) at (4.2,2.7) {$(Y_4,\rho_4)$};
 
    \node[draw, circle, fill=black, inner sep=1pt] (A) at (0,0) {};
    \node[draw, circle, fill=black, inner sep=1pt] (B) at (3.8,0.3) {};
    \node[draw, circle, fill=black, inner sep=1pt] (C) at (-0.5,1.6) {};
    \node[draw, circle, fill=black, inner sep=1pt] (D) at (3,2) {};

    \draw (A) -- (B) node[midway, below] {$3$};
    \draw (A) -- (C) node[midway, left] {$2$};
    \draw (A) -- (D) node[midway, right] {$3$};
    \draw (C) -- (D) node[midway, above] {$3$};
    \draw (B) -- (D) node[midway, right] {$2$};
    \draw (B) to[out=45,in=45,looseness=1.5] (C) node[midway,right] {};
  \end{scope}

\end{tikzpicture}
\caption{ $(X_4,d_4)$ and $(Y_4,\rho_4)$ are not ${\bf US}$-spaces by Theorem~\ref{xz}.\label{cis}}
\end{figure}

The next conjectures were formulated in \cite{Dov-Rov}.

\begin{conjecture}
\label{sepjtg}
The following statements are equivalent for every finite ultrametric space $(X^*,d^*)$:
\begin{enumerate}[label=\normalfont(\roman*), left=0pt] 
\item $(X^*, d^*) \notin {\bf US}.$

\item $(X^*, d^*)$  contains a four-point subspace which is weakly similar either to  $(X_4,d_4)$ or  to $(Y_4,\rho_4)$.
\end{enumerate}
\end{conjecture}

\begin{conjecture}
  \label{tghjnk}
    Let $(X,d)$ be an infinite ${\bf US}$-space generated by labeled star graph with a center $c$,
    let $X_0 := X \setminus \{c\}$, and let $d_0$ be the restriction of $d$ on the set $X_0\times X_0$.
    Then the following statements are equivalent:
\begin{enumerate}[label=\normalfont(\roman*), left=0pt] 
    \item $(X, d)$ is compact.

   \item $(X_0, d_0)$ is generated by labeled ray $R(l)$, $R=(x_1,x_2,\ldots, x_n,\ldots)$,
    such that
 \begin{equation*}
        \lim\limits_{n \to \infty} l(x_n) = 0
    \end{equation*}
and 
    \begin{equation*}
        l(x_n) \geq l(x_{n+1}) > 0
    \end{equation*}
    for every $n \in \mathbb{N}$.
    \end{enumerate}
\end{conjecture}

\section{Finite {\bf US}-spaces}

The main goal of the present section is to describe finite ${\bf US}$-spaces by the above four-point condition.

Let $(X, d)$ be an ultrametric space generated by the labeled star graph $S(l)$ with a center $c$ and let $Y$ be a subset $X$ such that 
\begin{equation}
\label{98822}
    c \in Y.
\end{equation}
Then using Definition~\ref{wer} and formula \eqref{e1.1} with  $T(l) = S(l)$
it is easy to see that 
\begin{equation}
    \label{drfy}
(Y, d|_{Y \times Y}) \in {\bf US},
\end{equation}
where $d|_{Y \times Y}$ is the restriction of the 
ultrametric $d$ on the set $Y \times Y$.

The following example shows that \eqref{drfy}
can be false if $c \notin Y$.

\begin{example}
Let $(X,d)$ be the ultrametric space generated by labeled star graph $S(l)$ depicted in Figure~\ref{ddfff} and let
\begin{equation*}
Y := X \setminus \{c\}.
\end{equation*}
Then, using Theorem~\ref{xz}, 
we obtain
\begin{equation*}
(Y, d|_{Y \times Y}) \notin {\bf US}.
\end{equation*}
\end{example}

\begin{figure}[H]
    \centering
    \begin{tikzpicture}[every node/.style={scale=0.9, inner sep=0pt, outer sep=0pt}]

        \node[circle,draw] (o) at (7,0) {0};

        \node[circle,draw,fill=white,inner sep=2pt] at (9,0) (n1) {$1$}; 
        \node[circle,draw,fill=white,inner sep=1pt] at (8.7,0.9) (n2) {$\frac{1}{2}$};
        \node[circle,draw,fill=white,inner sep=1pt] at (8.2,1.6) (n3) {$\frac{1}{3}$};
        \node[circle,draw,fill=white,inner sep=1pt] at (7.5,2.1) (n4) {};
        \node[circle,draw,fill=white,inner sep=1pt] at (7,2.3) (n5) {};
        \node[circle,draw,fill=white,inner sep=1pt] at (6.5,2.1) (n6) {$\frac{1}{n}$};
        \node[circle,draw,fill=white,inner sep=0.2pt] at (5.8,1.6) (n7) {$\frac{1}{n+1}$};
        \node[circle,draw,fill=white,inner sep=0.2pt] at (5.3,0.9) (n8) {$\frac{1}{n+2}$};
        \node[circle,draw,fill=white,inner sep=1pt] at (5,0) (n9) {};
        \node[circle,draw,fill=white,inner sep=1pt] at (5.3,-0.9) (n10) {};

        \draw (o) -- (n1);
        \draw (o) -- (n2);
        \draw (o) -- (n3);
        \draw[dashed] (o) -- (n4);
        \draw[dashed] (o) -- (n5);
        \draw (o) -- (n6);
        \draw (o) -- (n7);
        \draw (o) -- (n8);
        \draw[dashed] (o) -- (n9);
        \draw[dashed] (o) -- (n10);

        \node at (8.5,2.6) {$S(l)$};

    \end{tikzpicture}
    \caption{The unique point $x_0$ satisfying \eqref{zg} 
    for all $x,y\in Y$ which satisfy
     \eqref{cd}, is the center of the labeled star graph $S(l)$.\label{ddfff}}
\end{figure}
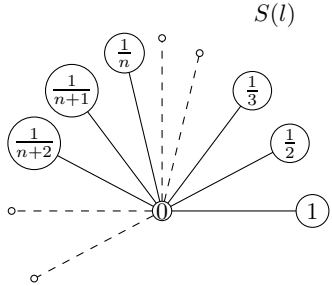

In the case of finite $Y$, condition \eqref{98822} is not necessary for membership \eqref{drfy}.

\begin{proposition} 
\label{8866gh}
Let  (X,d)  be an  {\bf US}-space.
Then $(Y, d|_{Y  \times Y})$ is also an ${\bf US}$-space for every finite non-empty $Y \subseteq X$.
\end{proposition}

\begin{proof}
 Let $S(l)$ be a labeled star graph generating $(X,d)$, let $c$ be the center of $S$, and let $Y$ be a finite non-empty subset of $X$. Since $Y$ is finite, there is a point $y^* \in Y$ such that 
 \begin{equation}
     \label{dgwqq}
     l( y^*) \leq l(y)
 \end{equation}
for each $y \in Y$. Now, using Definition~\ref{wer} and formula \eqref{e1.1} with $T(l) = S(l)$ we obtain the equality
\begin{equation*}
    d(y_i, y_j) = \max \{ l(c), l(y_i), l(y_j) \}
\end{equation*}
for all different $y_i, y_j \in Y$. The last equality and \eqref{dgwqq} give us
\begin{equation}
    d(y^*, y_i) \leq d(y_j, y_i)
\end{equation}
whenever $y^* \neq y_i \neq y_j$. Hence  $(Y, d|_{Y \times Y}) \in {\bf US}$ holds by Theorem~\ref{xz}. 

\end{proof}

The following theorem shows, in particular, that Conjecture~\ref{sepjtg} is true.

\begin{theorem}
\label{eeerlm}
Let  $(X,d)$ be a nonempty  ultrametric space. Suppose that either $X$ is finite, or $X$ has a limit point. Then  $(X,d)$ is an  $ {\bf US}$-space if and only if  $(X,d)$ contains no four-point subspace which is weakly similar to  $(X_4,d_4)$  or to $(Y_4,\rho_4)$.
\end{theorem}

\begin{proof}
The theorem is trivially true if 
\begin{equation*}
    \operatorname{card} X\leq 3.
\end{equation*}

Let us consider the case when 
\begin{equation*}
    \operatorname{card} X\geq 4.
\end{equation*}

Suppose that $(X,d) \in \mathbf{US}$. Then, by Theorem~\ref{xz} there is a point $x_0 \in X$ such that $d(x,y) = \max\{d(x,x_0), d(y,x_0)\}$ for all $x,y \in X$. Let us define $l : X \to \mathbb{R}^+$ by $l(x) = d(x, x_0)$ for every $x\in X$. For all distinct $x,y,z,w \in X$ we have  
\begin{align*}
d(x,y) &=\max\{l(x), l(y)\} \leq \max\{l(x), l(y),l(z), l(w)\} 
\\
&=\max\{d(x,z), d(y,w)\}
\end{align*}
and, in particular, $\{x,y,z,w\}$ is not weakly similar to $(X_4,d_4)$ or $(Y_4,\rho_4)$.  

Now, suppose that $(X,d) \notin \mathbf{US}$. Define $l : X \to \mathbb{R}^+$ by  
\begin{equation*}
    l(x) =\inf\limits_{y \neq x} d(x,y).
\end{equation*}

If $d(x,y) = \max\{l(x), l(y)\}$ for all distinct $x, y \in X$, we can choose $x_0 \in X$ such that $l(x_0)$ is minimal (if $X$ has a limit point $x_0$, $l(x_0)=0$ is minimal) and then $(X,d) \in \mathbf{US}$ by Theorem~\ref{xz}. Therefore, there are distinct $x, y \in X$ such that 
\begin{equation*}
    d(x,y) \neq \max\{l(x), l(y)\}.
\end{equation*}
Clearly $\max\{l(x), l(y) \}\leq d(x,y)$ by  definition of $l$. Since $x \neq y$, we have $\max\{l(x), l(y) \}< d(x,y)$. Therefore, there are $z, w \in X$ such that $z \neq x$, $d(x,z) < d(x,y)$, $w \neq y$, and $d(y,w) < d(x,y)$. It is also clear that the relations $z \neq y$ and $w \neq x$ hold, and $z \neq w$ since otherwise we would have $d(x,y) > \max\{d(x,z), d(z,y)\}$. Hence, $x,y,z,w$ are distinct. We have 
\begin{equation*}
    d(x,y) = d(x,w) = d(z,y) = d(z,w)
\end{equation*}
 by the strong triangle inequality. It follows that $\{x,y,z,w\}$ is weakly similar to either $(X_4,d_4)$ (if $d(x,z) \neq d(y,w)$) or to $(Y_4,\rho_4)$ (if $d(x,z) = d(y,w)$).  

The proof is complete.
\end{proof}

The following lemma easily follows from  Theorem~4.1 of \cite{Petrov}, see also Theorem~4.4 and Proposition~4.1 in \cite{Dovgoshey-Kostikov-2023}.

\begin{lemma}\label{bvcio}
 Let $(X,d)$ be a four-point ultrametric space.
Then the following statements are equivalent:

\begin{enumerate}[label=\normalfont(\roman*), left=0pt] 
    \item 
$(X,d)$ is generated by labeled tree.

 \item There is a point $x_0 \in X$ such that
\begin{equation*}
d(x_0, x) = \max \{ d(p,q) : p,q \in X \}
\end{equation*}
for each $x \in X\setminus\{x_0\}$.
\end{enumerate}
\end{lemma}

Theorem \ref{eeerlm} and Lemma \ref{bvcio} give us the following.

\begin{corollary}\label{aaattr}
    Let $(X, d)$ be a finite ultrametric space with
\begin{equation}
\label{7_…47}
\operatorname{card} X \leq 4.
\end{equation}

Then the following statements are equivalent:  

\begin{enumerate}[label=\normalfont(\roman*), left=0pt]
    \item 
$(X, d) \in {\bf US}.$  

\item There is a labeled tree $T(l)$ such that $(X, d)$ is generated by $T(l)$.  

\end{enumerate}

\end{corollary}

The next example shows that the number $4$ is the best possible constant in inequality \eqref{7_…47}.

\begin{example}
The five-point ultrametric spaces generated by labeled trees $T_1(l_1)$ and $T_2(l_2)$ (see Figure~\ref{dftthh} below) do not belong to ${\bf US}$.
\end{example}

\begin{figure}[H]
\centering
\begin{center}
    \begin{tikzpicture}[scale=1, every node/.style={circle, draw, minimum size=5mm}]
        \node (A) at (0,2) {2};
        \node (B) at (1.5,2) {2};
        \node (C) at (3,2) {3};
        \node (D) at (4.5,2) {1};
        \node (E) at (6,2) {1};

        \draw (A) -- (B) -- (C) -- (D) -- (E);

        \node[draw=none] at (8,2) {$T_1(l_1)$};

        \node (F) at (0,0) {2};
        \node (G) at (1.5,0) {2};
        \node (H) at (3,0) {3};
        \node (I) at (4.5,0) {2};
        \node (J) at (6,0) {2};

        \draw (F) -- (G) -- (H) -- (I) -- (J);

        \node[draw=none] at (8,0) {$T_2(l_2)$};

    \end{tikzpicture}
\end{center}
\caption{The ultrametric space $(V(T_1), d_{l_1})$ contains four-point subspace which is isometric to $(X_4, d_{4})$.  
Similarly, $(V(T_2), d_{l_2})$ contains a subspace isometric to $(Y_4, \rho_{4})$.}
\label{dftthh}
\end{figure}
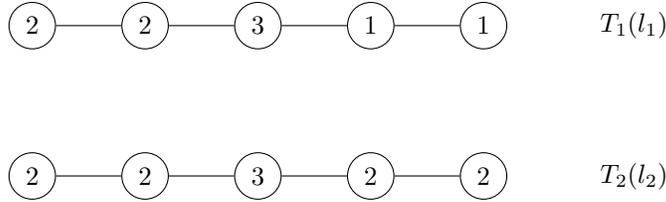

The next result is a semimetric modification of Theorem~\ref{eeerlm}.

\begin{theorem}\label{333sss}
    Let $(X, d)$ be a finite semimetric space with
    \begin{equation}\label{njjcexa}
        \operatorname{card}X\neq 3.
    \end{equation}
    Then the following statements are equivalent:

\begin{enumerate}[label=\normalfont(\roman*), left=0pt]
    \item 
$(X, d) \in {\bf US}$. 

\item
Each four-point subspace of $(X, d)$ is weakly similar to an ${\bf US}$-space.  

\item Each four-point subspace of $(X, d)$ is weakly similar to an ultrametric space generated by a labeled tree.
\end{enumerate}

\end{theorem}

\begin{proof}
 (i) $\Rightarrow$ (ii). The validity of this implication follows from Proposition \ref{8866gh}.

 (ii) $\Rightarrow$ (iii). Each labeled star graph is a labeled tree. 
Thus, implication (ii) $\Rightarrow$ (iii) 
is true.

 (iii) $\Rightarrow$ (i). Let (iii) hold. It follows from Definition \ref{wdvthk} and condition~\eqref{njjcexa} that  the semimetric space $(X,d)$ is ultrametric iff every four-point subspace of $(X,d)$ is ultrametric.

Therefore, by Proposition \ref{ooyggr}, each of statements (ii) and (iii) implies that $(X,d)$ is ultrametric. Now, using Corollary~\ref{aaattr}, we see that statements (ii) and (iii) are logically equivalent.

To complete the proof, it suffices to note that, for the ultrametric space $(X,d)$, statement (ii) implies statement (i) by Theorem~\ref{eeerlm}.

\end{proof}

The following example shows that restriction~\eqref{njjcexa} is really necessary.

\begin{example}
    Let $(X,d)$ be a three-point space consisting of the vertices of a right triangle on the Euclidean plane. Then statements (ii) and (iii) of Theorem~\ref{333sss} are vacuously true, but statement (i) of this theorem is false.
\end{example}

\section{Compact {\bf US}-spaces}

Let us start with a characterization of infinite
labeled star graphs $S(l)$ which generate compact {\bf US}-spaces.

\begin{theorem}\label{qyuiz}
Let $(X,d)$ be an infinite $\bf US$-space generated by labeled star graph $S(l)$ with a center $c$. Then the following statements are equivalent:
\begin{enumerate}[label=\normalfont(\roman*), left=0pt] 
    \item $(X,d)$ is compact.
    
    \item The equality 
    \begin{equation}
    \label{rvyhkij}
        l(c) = 0
    \end{equation}
    holds and the set 
    \begin{equation}
    \label{wswsolk}
        A_\varepsilon := \{ x \in X : l(x) \geq \varepsilon \}
    \end{equation}
    is finite for each $\varepsilon > 0$.
\end{enumerate}

\end{theorem}

\begin{proof}
    
(i)$ \Rightarrow$ (ii).  Let $(X,d)$ be compact.
If \eqref{rvyhkij} does not hold, then we have
\begin{equation*}
    l(c)=t_0
\end{equation*}
for some $t_0>0$.

Definition \ref{wer} and formula \eqref{e1.1} with $T(l) = S(l)$ give us the inequality
\begin{equation}
\label{gbuefi}
    d_l(u,v) \geq t_0 
\end{equation}
for all distinct $u,$ $v \in X$.

Let $(u_n)_{n \in \mathbb{N}}$ be an arbitrary sequence of distinct points of $X$. Since \eqref{gbuefi} holds for all distinct $u,$ $v \in X$, the sequence $(u_n)_{n \in \mathbb{N}}$ has no convergent subsequence.
Thus $(X,d)$ is not compact by Proposition \ref{graa}.

Suppose now that the set $A_{\varepsilon_0}$ is infinite for some $\varepsilon_0 > 0$. Then there is a sequence $(a_n)_{n \in \mathbb{N}}$ of distinct points $a_n \in A$. As above, using formula  \eqref{e1.1} and Definition  \ref{wer} we see that the inequality
\begin{equation*}
    d(u_n, u_m) \geq \varepsilon_0
\end{equation*}
holds for all distinct $n,$ $m \in \mathbb{N}$.
Thus, $(u_n)_{n \in \mathbb{N}}$ has no convergent subsequence, and, consequently, $(X,d)$ is not compact
by Proposition \ref{graa}.

(ii)$ \Rightarrow$ (i).  Let (ii) hold and let $ X_0 := X\setminus \{c\}$. Then equality \eqref{rvyhkij} and Theorem \ref{t1.4}
imply the inequality
\begin{equation}
\label{998dde}
l(x) > 0 \quad 
\end{equation}
for every $x \in X_0$. Let $(\varepsilon_n)_{n \in \mathbb{N}}$ 
be a strictly decreasing sequence of positive real numbers such that
\begin{equation}
\label{strict}
\lim\limits_{n \to \infty} \varepsilon_n = 0.
\end{equation}
Then from limit relation \eqref{strict} and inequality \eqref{998dde} follows
the equality
\begin{equation}
\label{23_56[7}
X_0 = \bigcup\limits_{n \in \mathbb{N}} A_{\varepsilon_n}. 
\end{equation}
Now using the finiteness of $A_{\varepsilon}$ and formulas
 \eqref{wswsolk}, \eqref{23_56[7} we see that $X_0$ is countably infinite 
and there is a bijection between $\mathbb N$ and $X_0$ given by a sequence $(x_n)_{n \in \mathbb{N}}$ such that
\begin{equation}
\label{777}
\lim\limits_{n \to \infty} l(x_n) = 0
\end{equation}
and 
\begin{equation}
\label{888}
0 < l(x_{n+1}) \leq l(x_n)
\end{equation}
for each $ n \in \mathbb{N}$.
Now, using \eqref{e1.1} and \eqref{rvyhkij} we obtain the equality
\begin{equation}
\label{8881}
d(c, x_n) = l(x_n)
\end{equation}
for every $ n \in \mathbb{N}$. Equalities \eqref{777} and \eqref{8881} give us
\begin{equation*}
    \lim\limits_{n \to \infty} d(c, x_n) = 0.
\end{equation*}
Thus $(x_n)_{n \in \mathbb{N}}$ is a convergent sequence in $(X, d)$.

Let us consider an arbitrary sequence $(y_m)_{m \in \mathbb{N}}$ of points of $X$.
According to Proposition \ref{graa}, to complete 
the proof if is sufficient to verify that $(y_m)_{m \in \mathbb{N}}$ contains a convergent subsequence.
The existence of such subsequence is obvious if the range of $(y_m)_{m \in \mathbb{N}}$ is finite.
Otherwise, using the finiteness of the set $A_{\varepsilon}$ for all $\varepsilon>0$, there exists a subsequence $(y_{m_k})_{k \in \mathbb{N}}$  of $(y_m)_{m \in \mathbb{N}}$ such that
\begin{equation*}
d(c, y_{m_k}) = l(y_{m_k})<\varepsilon_k
\end{equation*}
for all $k \in \mathbb{N}$. Equality \eqref{strict} then implies that
\begin{equation*}
    \lim\limits_{n \to \infty} d(c, y_{m_k}) = 0.
\end{equation*}
Thus $(y_{m_k})_{k \in \mathbb{N}}$ is a convergent subsequence of $(y_m)_{m \in \mathbb{N}}$.

The proof is completed.

\end{proof}

    



The following theorem is a corollary of Theorem \ref{eeerlm}.

\begin{theorem}\label{siti}
Let $(X,d)$ be a compact ultrametric space. Then $(X,d)$ is an $\bf US$-space if and only if $(X,d)$ contains no four-point subspace which is weakly similar to $(X_4, d_4)$ or to $(Y_4, \rho_4)$.
\end{theorem}

\begin{proof}
     Theorem \ref{eeerlm} holds for every nonempty ultrametric space, except the infinite discrete spaces. The compact space $(X, d)$ cannot be infinite and discrete. Thus, Theorem 4.3 gives us the desired result.

\end{proof}

The next theorem shows that Conjecture~\ref{tghjnk} is true.

\begin{theorem}\label{wop}
\label{tghjnk1}
    Let $(X,d)$ be an infinite ${\bf US}$-space generated by labeled star graph $S(l)$ with a center $c$,
    let $X_0 := X \setminus \{c\}$, and let $d_0$ be the restriction of $d$ on the set $X_0\times X_0$.
    Then the following statements are equivalent:
\begin{enumerate}[label=\normalfont(\roman*), left=0pt] 
    \item $(X, d)$ is compact.

   \item $(X_0, d_0)$ is generated by labeled ray $R(l^*)$, $R=(x_1,x_2,\ldots, x_n,\ldots),$
    such that
 \begin{equation*}
        \lim\limits_{n \to \infty} l^*(x_n) = 0
    \end{equation*}
and 
    \begin{equation*}
        l^*(x_n) \geq l^*(x_{n+1}) > 0
    \end{equation*}
    for every $n \in \mathbb{N}$.
    \end{enumerate}
\end{theorem}

\begin{proof}

    We will prove that conditions (i) and (ii) are both equivalent to a third condition (iii), which states that the set $\{x \in X : d(x,c) \geq \varepsilon\}$
    is finite 
    for every  $\varepsilon > 0$.  

Suppose that condition (iii) does not hold. Then there are a positive number $\varepsilon > 0$ and an infinite sequence $(x_n)_{n\in \mathbb N}$ of distinct points $x_n \in X$ such that $d(x_n,c) \geq \varepsilon$ for all $n \in \mathbb{N}$. It follows  from formula \eqref{e1.1} with $T(l)=S(l)$ that $d(x_n,y) \geq \varepsilon$ holds for all $n \in \mathbb{N}$ and $y \in X\setminus \{x_n\}$. Thus the sequence $(x_n)_{n\in \mathbb N}$
has no convergent subsequence, so condition (i) does not hold by Proposition \ref{graa}.

Suppose that condition (ii) holds. Let $\varepsilon > 0$. Then for all but finitely many $n \in \mathbb{N}$ we have $l^*(x_n) < \varepsilon$, since $\lim\limits_{n \to \infty} l^*(x_n) = 0$. It follows from Proposition \ref{fraas} and (ii) imply that $d(x_n,c) \leq d(x_n,x_{n+1}) = l^*(x_n)$ for all $n \in \mathbb{N}$, so condition (iii) holds.  

Suppose that condition (iii) holds. Every open ball around $c$ is cofinite, so every open set containing $c$ is cofinite. This easily implies condition (i). Every $Y \subseteq X$ also satisfies condition (iii), so for given  $y_* \in Y\setminus \{c\}$ the set   $\{y \in Y : d(y,c) \geq d(y_*,c)\}$ is finite. This implies that if $\emptyset \neq Y \subseteq X$, then there exists $y \in Y$ such that $d(y,c)$ is maximal.  
We define a sequence $x_n$ of points in $X_0$ inductively. For all $n \in \mathbb{N}$, let $x_n$ be a point $x \in X_0 \setminus \{x_1, \ldots, x_{n-1}\}$ such that $d(x,c)$ is maximal ($X_0 \setminus \{x_1, \ldots, x_{n-1}\} \neq \emptyset$, since $X$ is infinite). This is a sequence of distinct points by construction, so in particular its range is infinite.  
Given $x \in X_0$, the set 
$\{y \in X : d(y,c) \geq d(x,c)\}$ is finite, so $d(x_n,c) < d(x,c)$ for some $n \in \mathbb{N}$. Hence $x \in \{x_1, \ldots, x_{n-1}\}$ by the definition of $x_n$. Therefore, the equality $X_0 = \{x_n :n \in \mathbb{N}\}$ holds. Let us define $l^* : X_0 \to \mathbb{R}^+$ by $l^*(x) = d(x,c)$. We have the inequality $l^*(x_n) \geq l^*(x_{n+1})$ by the definition of $x_n$, and  $
    \lim\limits_{n \to \infty} l^*(x_n) = 0
$
follows from condition (iii) since the points $x_n$ are pairwise distinct. If $m < n$, then $d(x_m, x_n) = \max\{l^*(x_m), l^*(x_n)\}$ since $(X,d)\in \mathbf{US}$, so since $(l^*(x_n))_{n\in \mathbb N}$ is a decreasing sequence, we have  
$
    d(x_m, x_n) = \max\limits_{m \leq i \leq n} l^*(x_i).
$
Therefore, the labeled ray $R(l^*)$ generates $(X_0, d_0)$, so condition (ii) holds. 

The proof is completed.

\end{proof}

The following lemma is a reformulation of Proposition 3.6 from \cite{Dovgoshey-Vito}.

\begin{lemma}\label{erra1}
 Let $R = (x_1, x_2, \dots)$ be a ray with labeling $l^*: V(R)\to {\mathbb R}^+$ and let $(V({R}), d_{l^*})$ be an ultrametric space generated by labeling $R(l^*)$. Then
the following statements are equivalent:
\begin{enumerate}[label=\normalfont(\roman*), left=0pt] 
    \item The sequence $(x_n)_{n \in \mathbb{N}}$ is a Cauchy sequence in $(V({R}), d_{l^*})$.
    
    \item Each infinite sequence of points of $(V({R}), d_{l^*})$ contains a Cauchy subsequence.
\end{enumerate}

\end{lemma}

The next lemma is a direct corollary of Lemma~3.19 from \cite{BDKP}.

\begin{lemma}\label{salvad}
Let $(X,d)$ be a compact ultrametric space and let $Y$ be a compact subset of $X$.  
If $(X,d)$ is isometric to some subspace of the space $(Y,d|_{Y \times Y})$,  
then the equality $X = Y$ holds. 
\end{lemma}

The next theorem claims that every  compact $\bf US$-space is, up to isometry, the completion of an ultrametric space generated by a labeled ray.

Let $R = (x_1, x_2, \dots)$ be a ray with labeling $l^*: V(R) \to \mathbb{R}^+$.  
We say that $l^*$ is a {\it decreasing labeling} if  the sequence
$
    (l^*(x_n))_{n \in \mathbb{N}}
$
is decreasing.

\begin{theorem}\label{sofas}
     Let $(X, d)$ be an infinite ultrametric space. Then the following statements are equivalent:
\begin{enumerate}[label=\normalfont(\roman*), left=0pt] 
    \item $(X, d)$ is a compact $\bf US$-space.
    
    \item $(X, d)$ is the completion of totally bounded $X_0 \subsetneq  X$ generated by labeled ray with decreasing labeling.
\end{enumerate}

\end{theorem}

\begin{proof}

(i) $\Rightarrow$ (ii).  
Let $(X, d)$ be a compact {\bf US}-space, generated by labeled star graph $S(l)$ with a center $c$.  

As in Theorem \ref{tghjnk1} we write  
\begin{equation}
\label{thhjqqq}
    X_0 :=X \setminus \{ c \}\,\, {\rm and} \,\, d_0 := d|_{X_0 \times X_0}.
\end{equation} 

By Theorem \ref{tghjnk1} the space $(X_0, d_0)$ is generated by labeled ray $R(l^*)$,   
$
    R = (x_1, x_2, \ldots, x_n, \dots)
$
such that  
$
    l^*: V(R) \to \mathbb{R}^+
$
is decreasing and  
\begin{equation}
    \lim\limits_{n \to \infty} 
    l^*(x_n) = 0,
\end{equation}  
where  
\begin{equation}
\label{aqasw}
    l^*(x_n) = d(c, x_n)
\end{equation}  
for every $n \in \mathbb{N}$. Using \eqref{thhjqqq}--\eqref{aqasw} we see that $X_0$ is a dense subset of $(X, d)$.  
Moreover, $(X_0, d_0)$ is totally bounded as a subspace of compact space $(X, d)$.  
Since every compact ultrametric space is complete,  $(X, d)$ is a completion of totally bounded space $X_0  \subsetneq X$ generated by $R(l^*)$.  

(ii) $\Rightarrow$ (i).  
Let $(X, d)$ be the completion of  totally bounded $X_0\subsetneq X$ generated by labeled ray $R(l^*)$, $
    R = (x_1, x_2, \dots)
$
with decreasing labeling 
$
    l^*: V(R) \to \mathbb{R}^+.
$

Then the space $(X,d)$ is compact  by Proposition \ref{qru}. Hence to complete the proof it suffices to show that 
\begin{equation}
\label{poooo}
(X,d) \in {\bf US}.
\end{equation}
Proposition \ref{erra} and Lemma \ref{erra1} imply that $(x_n)_{n \in \mathbb{N}}$ is a Cauchy sequence in $(X,d)$.  
Since every Cauchy sequence of points of compact ultrametric space is a convergent sequence in this space,  
$(x_n)_{n \in \mathbb{N}}$ is a convergent sequence in $(X,d)$.

Let $x_0 \in X$ be a limit point of the sequence $(x_n)_{n \in \mathbb{N}}$, 
\begin{equation}
    \label{algs}
\lim\limits_{n\to\infty}x_n=x_0.
\end{equation}

 Then  the set $V(R) \cup \{x_0\}$ is a compact subset of $(X,d)$. 
Using Lemma \ref{salvad} we obtain the equality
\begin{equation}
\label{casy}
X = V(R) \cup \{x_0\}
\end{equation}
because $(X,d)$ is a compact space, and $V(R) \cup \{x_0\}$ is compact subset of $(X,d)$, and, by Definition \ref{ssmmcc}, $(X,d)$ is isometric to compact subset of the compact set $V(R) \cup \{x_0\}$.

Since the labeling $l^*:V(R) \to \mathbb{R}^+$ is decreasing, formula \eqref{e1.1} with $T(l) = R(l^*)$ gives us
\begin{equation}
\label{ymnvee}
d(x_n, x_m) = d_{l^*}(x_n, x_m) = \max \{ l^*(x_n), l^*(x_m) \} = l^*(x_{\min\{m,n\}})
\end{equation}
for all distinct $m,n \in \mathbb{N}$.

For every $p \in X$ the function 
\begin{equation*}
    X \ni x \mapsto d(x,p) \in \mathbb{R}^+
\end{equation*}
is a continuous mapping from $(X,d)$ to $\mathbb{R}^+$. Hence \eqref{algs} implies
\begin{equation}
\label{qweepmn}
\lim_{n \to \infty} d(x_n, p) = d(x_0, p)
\end{equation}
for each $p \in X$. In particular, using \eqref{ymnvee} with $m = n_1$ and \eqref{qweepmn} with $p = x_{n_1}$ we obtain
\begin{align*}
d(x_0, x_{n_1}) =& \lim_{n \to \infty} d(x_n, x_{n_1})\\
=& \lim_{n \to \infty} l^*(x_{\min\{n_1,n\}})=l^*(x_{n_1})\\
\leq&\max \{ l^*(x_{n_1}), l(x_{n}) \} =d(x_{n_1}, x_{n}),
\end{align*}
wherever $n \in \mathbb{N}$ and $n \neq n_{1}$. Thus the inequality
\begin{equation}
\label{4tru}
d(x_0, x) \leq d(y, x)
\end{equation}
holds whenever $x, y \in V(R) \cup \{x_{0}\}$ and $x_0 \neq x \neq y$.

Membership \eqref{poooo} follows from \eqref{casy} and \eqref{4tru} by Theorem \ref{xz}.

The proof is completed.
    
\end{proof}

\begin{example}
    \label{fgjkdfb}
Let us define an ultrametric $d^+\colon \mathbb{R}^+ \times \mathbb{R}^+ \to \mathbb{R}^+$ as
\begin{equation*}
d^+(p, q) =
\begin{cases}
0, & \text{if } p = q, \\
\max \{p, q\}, & \text{if } p \neq q.
\end{cases}
\end{equation*}
In \cite{Dov-Rov} it was noted  that $({\mathbb R}^+,d^+)$ is an ${\bf US}$-space.

\end{example}

Using Theorem \ref{wop} we can prove that an infinite subset \( A \) of \( \mathbb{R}^+ \) is compact subset of \( (\mathbb{R}^+, d^+) \) iff 
$A=\{t_n:n\in\mathbb{N}\}\cup\{0\}$ where \( (t_n)_{n \in \mathbb{N}} \)
 is a strictly decreasing sequence \( (t_n)_{n \in \mathbb{N}} \) of points of \( \mathbb{R}^+ \) such that the limit relation
\begin{equation*}
\lim_{n \to \infty} t_n = 0
\end{equation*}
holds
in the usual Euclidean topology.

\begin{remark}
The ultrametric $d^+$ on $\mathbb{R}^+$ was introduced by Delhomme, Laflamme, Pouzet, and Sauer in \cite{Delhommé}. 
In \cite{Ishiki} Yoshito Ishiki wrote:
"The space $ (\mathbb{R}^+, d^+) $ is as significant for ultrametric spaces as the space $ \mathbb{R}^+$ or $ \mathbb{R} $ with the Euclidean topology in the theory of usual metric spaces."
Some results related to the ultrametric space $ (\mathbb{R}^+, d^+) $ can be found in \cite{Dobrush,Dovgoshey-Kostikov-2024,Ishiki2023,Ishiki2021}.
\end{remark}

\section{Two conjectures}

We believe that the following hypothesis is correct.

\begin{conjecture}\label{mnb}
Let $(X,d)$ be an infinite ultrametric space. Then the following statements are equivalent:

\begin{enumerate}[label=\normalfont(\roman*), left=0pt]
    \item There is $(X^*,d^*) \in {\bf US}$ such that $(X,d)$ is isometric to a subspace of $(X^*,d^*)$.
    
    \item $(X,d)$ contains no four-point subspace which is weakly similar to $(X_4,d_4)$ or to $(Y_4,\rho_4)$.
\end{enumerate}

\end{conjecture}

We note that Conjecture \ref{mnb} is true by Theorem~\ref{siti} if $(X,d)$ is compact. Moreover, using Proposition~\ref{8866gh}, it is easy to prove the validity of implication (i) $\Rightarrow$ (ii) for arbitrary infinite $(X,d)$. Thus, to prove Conjecture~\ref{mnb} it suffices to show that (ii) $\Rightarrow$ (i) is valid for non-compact ultrametric spaces $(X,d)$.

The next conjecture gives us a partial generalization of Theorem~\ref{sofas}.

Recall that a tree \( T \) is {\it starlike} if it has exactly one vertex with degree greater than 2. (See, for example, \cite{Bu-Zhou,Lepovic-Gutman}).

\begin{conjecture}
    \label{jak2}
Let $(X,d)$ be an infinite ultrametric space. Then the following statements are equivalent.

\begin{enumerate}[label=\normalfont(\roman*), left=0pt] 
    \item $(X,d)$ is the completion of totally bounded $X_0  \subsetneq X$ generated by a labeled ray.
    
    \item There is a starlike rayless tree $T$ with a labeling $l: V(T) \to \mathbb{R}^+$ such that 
   $(X,d)$ is a compact ultrametric space generated by $T(l)$.
\end{enumerate}

\end{conjecture}

If $T$ is a star graph, then $T$ is starlike and rayless, but not vice versa, see, for example, Figure~\ref{unik} below.

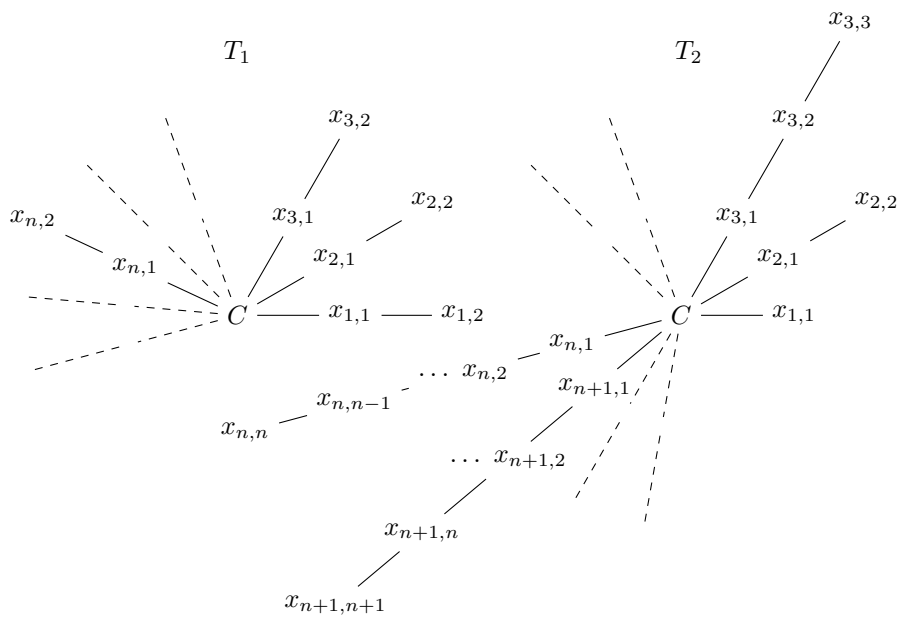
\begin{figure}[H]
    \centering
\begin{tikzpicture}

    \node (C1) at (-2,0) {$C$}
        child[grow=0] {node {$x_{1,1}$}
            child[grow=0] {node {$x_{1,2}$}}
        } 
        child[grow=30] {node {$x_{2,1}$}
            child[grow=30] {node {$x_{2,2}$}}
        } 
        child[grow=60] {node {$x_{3,1}$}
            child[grow=60] {node {$x_{3,2}$}}
        }
        child[grow=110, dashed] {node {} child[grow=110, dashed] {node {}}}  
        child[grow=135, dashed] {node {} child[grow=135, dashed] {node {}}}  
        child[grow=155] {node {$x_{n,1}$}
            child[grow=155] {node {$x_{n,2}$}}
        }
        child[grow=175, dashed] {node {} child[grow=175, dashed] {node {}}}  
        child[grow=195, dashed] {node {} child[grow=195, dashed] {node {}}};  
    \node[draw=none] at (-2,3.5) {$T_1$};

    \node (C2) at (3.9,0) {$C$}
        child[grow=0] {node {$x_{1,1}$}} 
        child[grow=30] {node {$x_{2,1}$}
            child[grow=30] {node {$x_{2,2}$}}
        } 
        child[grow=60] {node {$x_{3,1}$}
            child[grow=60] {node {$x_{3,2}$}
                child[grow=60] {node {$x_{3,3}$}}
            }
        }
        child[grow=110, dashed] {node {} child[grow=110, dashed] {node {}}}  
        child[grow=135, dashed] {node {} child[grow=135, dashed] {node {}}}  
        child[grow=195] {node {$x_{n,1}$}
            child[grow=195] {node {$\dots \,\,x_{n,2}$}
                child[grow=195] {node {$x_{n,n-1}$}
                    child[grow=195] {node {$x_{n,n}$}}
                }
            }
        }
        child[grow=220] {node {$x_{n+1,1}$}
            child[grow=220] {node {$\ldots \,\,x_{n+1,2}$}
                child[grow=220] {node {$x_{n+1,n}$}
                    child[grow=220] {node {$x_{n+1,n+1}$}}
                }
            }
        }
        child[grow=240, dashed] {node {} child[grow=240, dashed] {node {}}}  
        child[grow=260, dashed] {node {} child[grow=260, dashed] {node {}}};  
    \node[draw=none] at (4,3.5) {$T_2$};
\end{tikzpicture}

    \caption{Trees $T_1$ and $T_2$ are starlike and rayless.}
    \label{unik}
\end{figure}

\section*{Declarations}

\subsection*{Declaration of competing interest}

 The authors declare no conflict of interest.

\subsection*{Data availability}

  All necessary data are included into the paper.

\subsection*{Acknowledgement}

Oleksiy Dovgoshey was supported by grant $359772$ of the Academy of Finland.\\
Omer Cantor was supported by grant  $721/24$ of the Israeli Science Foundation.


\bigskip

CONTACT INFORMATION

\medskip
Oleksiy Dovgoshey\\
Department of Function Theory, Institute of Applied Mathematics and Mechanics of NASU, Slovyansk, Ukraine,\\
Department of Mathematics and Statistics, University of Turku, Turku, Finland \\
oleksiy.dovgoshey@gmail.com, oleksiy.dovgoshey@utu.fi

\medskip
Omer Cantor\\
Department of Mathematics, University of Haifa, Haifa, Israel\\
ocantor@proton.me

\medskip
Olga Rovenska\\
Department of Mathematics and Modelling, Donbas State Engineering Academy, Kramatorsk, Ukraine\\
rovenskaya.olga.math@gmail.com

\end{document}